\documentclass[12pt]{amsart}
\usepackage{statmath}

\usepackage[utf8]{inputenc}
\usepackage {amssymb}
\usepackage {amsmath}
\usepackage{amsthm}
\usepackage{graphicx}
\usepackage {amscd}
\usepackage {epic}
\usepackage{xypic}
\usepackage[colorlinks, citecolor=blue]{hyperref}
\usepackage[left=1.5in,right=1.5in,top=1.5in,bottom=1.5in]{geometry}
\usepackage{soul}
\usepackage{comment}
\usepackage{enumitem}
\usepackage{mathrsfs}

\theoremstyle{plain}
\newtheorem{theorem}{Theorem}[section]
\newtheorem{corollary}[theorem]{Corollary}
\newtheorem{lemma}[theorem]{Lemma}
\newtheorem{conjecture}[theorem]{Conjecture}

\newtheorem{definition-lemma}[theorem]{Definition-Lemma}
\theoremstyle{remark}
\newtheorem{remark}[theorem]{Remark}

\theoremstyle{definition}

\newcommand{\Pic}[0]{\operatorname{Pic}}

\def\BC{\operatorname{BC}}
\def\NA{\overline{\operatorname{NA}}}

\def\>{\geq}
\newcommand{\mbQ}{\mathbb{Q}}
\newcommand{\mbR}{\mathbb{R}}

\def\mcO{\mathcal{O}}
\newcommand{\num}{\equiv}

\newcommand{\OO}{{\mathcal{O}}}
\newcommand{\Q}{{\mathbb{Q}}}
\newcommand{\C}{{\mathbb{C}}}
\newcommand{\R}{{\mathbb{R}}}

\newcommand{\Supp}{{\rm Supp}}

\newcommand{\mbC}{\mathbb{C}}
\newcommand{\mbZ}{\mathbb{Z}}

\newcommand{\vphi}{\varphi}

\def\oomega{\boldsymbol{\omega}}

\def\ggamma{\boldsymbol{\gamma}}

\def\bbeta{\boldsymbol{\beta}}

\def\>{\geq}

\def\mcO{\mathcal{O}}

\def\mcC{\mathcal{C}}

\def\mcH{\mathcal{H}}
\def\mcK{\mathcal{K}}

\def\eps{\epsilon}

\def\pt{\operatorname{pt}}

\def\dim{\operatorname{dim}}

\def\NA{\operatorname{\overline{NA}}}

\def\Nef{\operatorname{Nef}}
\def\Im{\operatorname{Im}}
\def\Re{\operatorname{Re}}

\def\Int{\operatorname{Int}}
\def\Alb{\operatorname{Alb}}

\theoremstyle{definition}
\newtheorem{definition}[theorem]{Definition}
\theoremstyle{definition}

\numberwithin{equation}{section}

\theoremstyle{remark}

\author{Omprokash Das}
\address{School of Mathematics\\
Tata Institute of Fundamental Research\\
Homi Bhabha Road, Navy Nagar\\
Colaba, Mumbai 400005}
\email{omdas@math.tifr.res.in}
\email{omprokash@gmail.com}

\author{Christopher Hacon}
\address{Department of Mathematics\\
University of Utah\\
155 S 1400 E\\
Salt Lake City, Utah 84112}
\email{hacon@math.utah.edu}
\thanks{Christopher Hacon was partially supported by NSF research grants no: DMS-1952522, DMS-1801851, DMS-2301374 and by a grant from the Simons Foundation SFI-MPS-MOV-00006719-07. We would like to thank M. Paun, V. Tosatti, and an anonimous referee for their comments and suggestions.}

\begin{document}

\title[Transcendental MMP for Projective Varieties]{Transcendental Minimal Model Program for Projective Varieties}
\maketitle

\begin{abstract}
In this article we prove that if $(X,B+\bbeta )$ is a projective generalized klt pair such that $B+\bbeta_X$ is big, then 
$(X,B+\bbeta )$ admits a good minimal model or Mori fiber space. In particular, this implies Tossati's transcendental base-point-free conjecture for projective manifolds.

  %  In this  article we show that if $(X, B)$ is a klt pair, $X$ is a projective variety over $\mbC$ and $\alpha\in H^{1,1}_{\BC}(X)$ is a nef class such that $\alpha-(K_X+B)$ is K\"ahler, then there is a projective morphism with connected fibers to a normal projective variety $Y$ with rational singularities and a K\"ahler class $\omega_Y\in H^{1,1}_{\BC}(X)$ such that $\alpha=f^*\omega_Y$. As an application we show that on a projective manifold a K\"aler-Ricci flow has finite time extinction if and only if it is a Fano manifold.
\end{abstract}

\tableofcontents

\section{Introduction}
The purpose of this paper is to prove the minimal model program for projective generalized klt pairs (see Definition \ref{subs:g-pair}). A typical generalized klt pair $(X,B+N)$ consists of a smooth projective variety $X$, a divisor $B=\sum b_iB_i$ with simple normal crossings support and coefficients $0<b_i <1$, and a nef $\R$-Cartier divisor $N$. These pairs play an important role in the birational classification of projective varieties and especially in Kawamata's canonical bundle formula, Kawamata subadjunction, and the boundedness of Fano varieties (see \cite{Filipazzi20}, \cite{Bir21}).
In the K\"ahler context, it is however important and more natural to consider a more flexible version where  the nef
divisor $N$ is replaced by a nef (1,1) form $\beta$. If $X$ is not projective, there are in fact typically very few divisor classes on $X$ (and in particular no ample divisors), however there are usually plenty of $(1,1)$ classes in the Bott-Chern cohomology group $H^{1,1}_{\rm BC}$ (including K\"ahler classes).
Just as in the projective case, the geometry of a K\"ahler variety $X$ is controlled to a large extent by several cones in $H^{1,1}_{\rm BC}(X)$ such as the nef, K\"ahler, pseudo-effective, and big cones. In particular, important results on the K\"ahler-Ricci flow are informed by results about generalized pairs. To this end we have the following conjecture (see \cite[Question 5.5]{Tos12}, \cite[Conjecture 1.2]{FT18}, and \cite[Conjecture 3.6]{Tos23}) which is essential in understanding the finite time singularities of the K\"ahler Ricci flows on $X$.
\begin{conjecture}\label{c-tossati} 
Let $X$ be a compact K\"ahler manifold and $\alpha \in H^{1,1}_{\rm BC}(X)$ a nef class such that $\alpha -K_X$ is nef and big, then $\alpha $   is semiample (i.e. there is a holomorphic map $g:X\to Z$ and a K\"ahler form $\omega _Z$ such that $\alpha \equiv g^*\omega _Z$). 
\end{conjecture}
Note that when $\alpha$ is the class of an $\R$-divisor, then the conjecture is known by the base point free theorem (see eg. \cite{KM98}) and when $\dim X\leq 3$, this is a consequence of the generalized minimal model program established in \cite{DHY23}.
The main result of this paper is the following.
\begin{theorem}\label{t-main}
    Let $(X,B+\bbeta)$ be a compact K\"ahler generalized klt pair such that $B+\bbeta _X $ is big. Assume that $X$ is a $\mbQ$-factorial projective variety over $\mbC$. Then the following hold: 
    \begin{enumerate}
        \item if $K_X+B+\bbeta_X$ is pseudo-effective, then $(X,B+\bbeta)$ has a good minimal model \textcolor{blue}{(see \S \ref{subs:g-pair})}, and
        \item if $K_X+B+\bbeta_X$ is not pseudo-effective, then $(X,B+\bbeta)$ has a Mori fiber space.
    \end{enumerate}
\end{theorem}
In particular, this result implies Conjecture \ref{c-tossati} for projective varieties with mild singularities (see Corollary  \ref{cor:tossati-conj}).
It should be mentioned that generalized pairs were formally introduced in \cite{DHY23},
however, related results were quite common in the literature (see for example \cite{Gue20} and \cite{CH20}).
Note that when $\bbeta =0$, Theorem \ref{t-main} is the main result of \cite{BCHM10}. When $\bbeta $ is the class of a nef b-divisor, the result also follows from  \cite{BCHM10} by standard arguments.
In the general case we make use of the \emph{Graf-Kirschner decomposition} (see \S \ref{subs:GK-decomposition}) which allows us to write (up to some birational modification) $\bbeta_{X} \equiv N+\delta $ where $N$ is a nef $\R$-divisor and $\delta \in H^{1,1}_{\rm BC}(X)$ satisfies $\delta \cdot C=0$ for any curve $C$ on $X$.
We then show that a good minimal model for the generalized klt pair $(X,B+N)$ is also a good minimal model for the generalized klt pair $(X,B+\bbeta )$.

\section{Preliminaries}
We begin by fixing some notation. Let $X$ be a normal compact analytic variety. We let $N^1(X):=H^{1,1}_{\rm BC}(X)$ the Bott-Chern cohomology group and $N_1(X)$ be the dual space of real closed bi-dimension $(1,1)$ currents modulo the equivalence relation: $T_1\num T_2$ if and only if $T_1(\eta)=T_2(\eta)$ for all $\eta\in H^{1,1}_{\BC}(X)$. Let ${\Nef}(X)\subset N^1(X)$ (resp. $\mcK(X)\subset N^1(X)$) denote the nef cone (resp. K\"ahler cone), and $\NA(X)\subset N_1(X)$ the generalized Mori cone. See \cite[Definition 2.2]{DH20} for more details.  %If $X$ has rational singularities, then $\overline{\rm NA}(X)$ and ${\rm Nef}(X)$ are dual cones. 
If $X$ is a projective variety and $D$ is an $\mbR$-Cartier divisor, then from \cite[Corollary 0.2]{DP04} it follows that the class $[D]\in H^{1,1}_{\rm BC}(X)$ is nef (resp. K\"ahler) if and only if the divisor $D$ is nef in the usual sense (resp. ample), i.e. $D\cdot C\geq 0$ for every curve $C\subset X$ (resp. $0\ne D=\sum r_iA_i$ where $A_i$ is very ample and $r_i>0$).

% A Weil divisor $D$ (not necessarily $\mbQ$-Cartier) on a normal compact analytic variety $X$ is called \emph{big}, if there is a positive real number $C>0$ such that $h^0(X, \mcO_X(mD))>C\cdot m^{\dim X}$ for all sufficiently large and divisible $m\in \mbZ^+$. An $\mbR$-divisor $D$ is called \emph{big} if it is positive $\mbR$-linear combination of big divisors. A closed positive bi-degree $(1,1)$ current $T$ (not necessarily with local potentials) is called \textit{big}, if there is a K\"ahler from $\omega$ and a closed positive bi-degree $(1,1)$ current $T'$ such that $T=\omega+T'$. We also note that if $D=\sum r_i D_i$ is an effective $\mbR$-divisor on a normal analytic variety, then it naturally defines a closed positive bi-degree $(1,1)$ current $T_D$ (not necessarily with local potentials) in the following way: if $\mu:Y\to X$ is a resolution of singularities of $X$ and $D':=\sum r_iD_i'$ is the strict transform of $D$ under $\mu$, then $T_{D'}:=\sum r_i\int_{D'_i} $ is a closed positive bi-degree $(1,1)$ current on $Y$, and we define $T_D=\mu_*T_{D'}$. Passing to a Hironaka's hat diagram it is easy to see that this definition is independent of the resolution of $\mu$. Finally, we also note that if $\omega$ is a K\"ahler form and $\Delta$ is an effective $\mbR$-divisor on a normal compact analytic variety $X$, then $T:=\omega+\Delta$ is a big current according our definition above. 

 Let $T$ be a closed positive bi-degree $(1,1)$ current with local potentials and $f:X'\to X$ a proper bimeromorphic morphism. Then the pullback $f^*T$ can be defined in the following way: there is an open cover $\{U_i\}$ of $X$ and \emph{psh} functions $\vphi_i:U_i\to \mbR\cup\{-\infty\}$ such that $T|_{U_i}=dd^c\vphi_i$. Then $f^*\vphi_i:=\vphi_i\circ f$ is a psh function on $U'_i:=f^{-1}U_i$. We then define $f^*T$ by $f^*T|_{U'_i}:=dd^c(\vphi_i\circ f)$.

\begin{lemma}\label{lem:kahler-mori-pushforward}
    Let $f:X\to Y$ be a proper morphism with connected fibers between normal compact K\"ahler varieties with rational singularities. Then $f_*\NA(X)=\NA(Y)$. 
\end{lemma}

\begin{proof}
   Replacing $X$ by a resolution we may assume that $X$ is a compact K\"ahler manifold.  Clearly $f_*\NA(X)\subset \NA(Y)$. For the converse, by contradiction assume that $f_*\NA(X)$ is a proper subcone of of $\NA(Y)$. Since $N_1(X)$ is a finite dimensional vector space, there is a functional $\lambda:N_1(Y)\to \mbR$ such that $\lambda([T])\geq 0$ for all $[T]\in f_*\NA(X)$, and $\lambda([T_0])<0$ for some $T_0\in\NA(Y)$. Then by \cite[Proposition 3.9]{HP16}, there is a real closed smooth $(1,1)$ form $\eta$ with local potentials on $Y$ such that $\lambda([T])=T(\eta)$ for all $[T]\in N_1(Y)$. Now define $\hat\lambda:N_1(X)\to \mbR$ by $\hat\lambda([\hat T])=\hat T(f^*\eta)$ for all $[\hat T]\in N_1(X)$. Thus $\hat\lambda([\hat T])=\hat T(f^*\eta)=f_*(\hat T)(\eta)=\lambda([f_*\hat T])\geq 0$ for all $[\hat T]\in \NA(X)$. This implies from \cite[Corollary 0.3]{DP04} that $[f^*\eta]$ is a nef class. Thus by \cite[Lemma 2.38]{DHP22} it follows that $[\eta]$ is a nef class, and hence  $\lambda$ is non-negative on $\NA(Y)$, this is a contradiction.  
\end{proof}

\begin{lemma}\label{lem:nef-kahler-cones}
    Let $X$ be a normal compact K\"ahler variety, and $\mcK(X)$ and $\Nef(X)$ are the K\"ahler and nef cone respectively. Then 
    \[\Nef(X)=\overline{\mcK(X)} \quad \mbox{and}\quad \Int(\Nef(X))=\mcK(X).\] 
\end{lemma}

\begin{proof}
This is well known, for example see the proof of \cite[Proposition 6.1(iii)]{Dem92}.
% \footnote{CH: I would remove this proof.}
%     This is well known to the experts, however, since $X$ is singular, we will sketch the argument here. Clearly $\overline{\mcK(X)}\subset \Nef(X)$. For the converse, choose a nef class $[\alpha]\in \Nef(X)$ and fix a K\"ahler form $\omega$. Since $X$ is compact, any two positive definite smooth hermitian forms are equivalent, and thus for each $\ve>0$ there is a smooth representative $\alpha_{\ve}\in [\alpha]$ such that $\alpha_{\ve}\geq -\ve\omega$. Then $\alpha_{\ve}+2\ve\omega\geq \ve \omega>0$, hence $[\alpha_{\ve}+2\ve\omega]\in\mcK(X)$ for all $\ve>0$. Taking limit as $\ve\to 0^+$ we see that $\Nef(X)\subset\overline{\mcK(X)}$. 

%     Again, let $[\alpha]$ be in the interior of $\Nef(X)$. Then for any K\"ahler form $\omega$, $[\beta_\epsilon]:=[\alpha]-\epsilon[\omega]$ is still in the interior for all $0<\epsilon\ll 1$, i.e. $[\beta_\epsilon]$ is a nef class. Therefore $[\alpha]=\epsilon[\omega]+\beta_\epsilon$ is a K\"ahler class, i.e. ${\rm Int}(\Nef(X))\subset \mcK(X)$. The reverse inclusion is obvious due to the openess of the K\"ahler cone. 
\end{proof}

\begin{lemma}\label{lem:duality}
   Let $X$ be a normal compact K\"ahler variety with rational singularities. Then $\Nef(X)$ and $\NA(X)$ are dual to each other via the canonical isomorphism  $\Phi: N^1(X)\to N_1(X)^*$ as in \cite[Proposition 3.9]{HP16}. In particular, a class $\omega\in H^{1,1}_{\BC}(X)$ is K\"ahler if and only if $\omega\cdot \gamma>0$ for all $\gamma\in\NA(X)\setminus\{0\}$. 
\end{lemma}

\begin{proof}
 The duality $\Nef(X)=\NA(X)^*$ follows from the same argument as in the proof of \cite[Proposition 3.15]{HP16} using \cite[Lemma 2.38]{DHP22} in place of \cite[Lemma 3.13]{HP16}.

 Now, since $X$ is K\"ahler, $\Nef(X)$ and $\NA(X)$ are strongly convex closed cones, and thus from \cite[Lemma 6.7(a)]{Deb01} it follows that \[\mcK(X)={\rm Int}(\Nef(X))={\rm Int}(\NA(X)^*)=\] \[\{\omega\in H^{1,1}_{\BC}(X)\;|\; \omega\cdot \gamma>0 \mbox{ for all } \gamma\in \NA(X)\setminus\{0\} \}.\] 
\end{proof}

\begin{definition}\label{def:bott-chern}
  Let $X$ be reduced connected complex space. We define the sheaf $\mcH_X$ on $X$ as follows: for any open subset $U\subset X$, $\mcH_X(U)$ is composed of the real parts of holomorphic functions in $\mcO_X(U)$ multiplied by $i$, where $i^2=-1$.

 Using the Cauchy-Riemann equation we get the following usual short exact sequence of sheaves on $X$
 \[\xymatrixcolsep{3pc}\xymatrix{
 0\ar[r] & \mbR\ar[r]^i & \mcO_X\ar[r]^{i\Re(\cdot)} & \mcH_X\ar[r] & 0.
 }
 \]
\end{definition}
 % \hl{Note that our definition is slightly more general than that of} \cite[Definition 3.1]{HP16} \hl{as we do not assume $X$ is reduced or irreducible.}\footnote{Om: Looks like this flexibility is no longer needed in the new proof, so maybe we should remove it? CH: yes}   

\begin{lemma}\label{lem:real-part}
   Let $f:X\to Y$ be a proper surjective morphism of analytic varieties with $f_*\mcO_X=\mcO_Y$. Then $f_*\mcH_X=\mcH_Y$.  
\end{lemma}

\begin{proof}
Let $U\subset Y$ be an open subset and $\vphi\in \mcH_X(f^{-1}U)$. Then there is a holomorphic function $u\in\mcO_X(f^{-1}U)$ such that $\vphi=i\Re(u)$. Since $f_*\mcO_X=\mcO_Y$, there is a holomorphic function $w\in \mcO_Y(U)$ such that $u=w\circ f$. Thus $\vphi=i\Re(w\circ f)=i(\Re(w)\circ f)$ and hence $f_*\mcH_X=\mcH_Y$.  
\end{proof}

\begin{lemma}\label{lem:trivial-pullback}
 Let $f:X\to Y$ be a proper morphism with connected fibers between normal compact analytic varieties with rational singularities. Assume that one of the following  two conditions hold:
 \begin{enumerate}
     \item $X$ and $Y$ are in Fujiki's class $\mcC$ and $f$ is bimeromorphic, or
     \item there is an effective $\mbQ$-divisor $B\geq 0$ such that $(X, B)$ is klt, $-(K_X+B)$ is $f$-nef-big and $f$ is projective. 
 \end{enumerate}
 Then $f^*:H^{1,1}_{\BC}(Y)=H^1(Y, \mcH_Y)\to H^{1,1}_{\BC}(X)=H^1(X, \mcH_X)$ and $f^*:H^2(Y, \mbR)\to H^2(X, \mbR)$ are both injective, and  
 \begin{align*}
      \Im(f^*)=\{\alpha\in H^{1,1}_{\BC}(X)\;|\; \alpha\cdot C &=0 \mbox{ for all curves } C\subset X \mbox{ s.t. } f(C)=\pt\}\\
                                                               & {\rm and}\\
        \Im(f^*)=\{\alpha\in H^2(X, \mbR)\;|\; \alpha\cdot C &=0 \mbox{ for all curves } C\subset X \mbox{ s.t. } f(C)=\pt\}.                                                       
 \end{align*}

\end{lemma}

\begin{proof}
    (1) This is \cite[Lemma 3.3]{HP16}.\\

    (2) The proof of this case is almost identical to that of \cite[Lemma 3.3]{HP16} using relative Kawamata-Viehweg vanishing theorem, however, for the convenience of the readers we add a brief argument below. 
    
    From Lemma \ref{lem:real-part} we have $f_*\mcH_X=\mcH_Y$. Thus from the Leray spectral sequence $E^{i,j}_2=H^i(Y, R^jf_*\mcH_X)$ we get the following exact sequence:
    \begin{equation}\label{eqn:h11-exact}
        0\to H^1(Y, \mcH_Y)\to H^1(X, \mcH_X)\to H^0(Y, R^1f_*\mcH_X).
    \end{equation}
    In particular, $f^*:H^{1,1}_{\BC}(Y)\to H^{1,1}_{\BC}(X)$ is an injection.

    Now since $-(K_X+B)$ is $f$-nef-big, by the relative Kawamata-Viehweg vanishing theorem (see \cite[Theorem 2.21]{DH20}), $R^if_*\mcO_X=0$ for all $i>0$. Therefore applying $f_*$ to the exact sequence 
    \begin{equation}\label{eqn:hx}
        0\to \mbR\to \mcO_X\to \mcH_X\to 0 
    \end{equation}
     and passing to the corresponding long exact sequence we get that $R^1f_*\mbR=0$ and $R^1f_*\mcH_X\cong R^2f_*\mbR$.
From the Leray spectral sequence $E^{i,j}_2=H^i(Y, R^jf_*\mbR)$ we get the following exact sequence
\begin{equation}\label{eqn:h2-exact}
    0\to H^2(Y, \mbR)\to H^2(X, \mbR)\to H^0(Y, R^2f_*\mbR)
\end{equation}

Combining all of these we obtain the following commutative diagram.
\begin{equation}\label{eqn:final-diagram}
    \xymatrixcolsep{3pc}\xymatrixrowsep{3pc}\xymatrix{
0\ar[r] & H^1(Y, \mcH_Y)\ar[r]^{f^*}\ar@{^{(}->}[d] & H^1(X, \mcH_X)\ar[r]\ar@{^{(}->}[d] & H^0(Y, R^1f_*\mcH_X)\ar[d]^{\cong}\\
 0\ar[r] & H^2(Y, \mbR)\ar[r]^{f^*} & H^2(X, \mbR)\ar[r]^\delta & H^0(Y, R^2f_*\mbR)   
    }
\end{equation}
 Since $R^1f_*\OO_X = 0$, $X$ and $Y$ have rational singularities, and $f$ is a projective morphism, we deduce that \cite[Theorem 12.1.3]{KM92} holds and hence so does condition $(\star)$ of the proof of \cite[Theorem 12.1.3]{KM92}. It follows that if $S\in H^2(X, \mbR)$ such that $\delta(S)\neq 0$, there is an element $\Gamma\in H_2(X/Y, \mbR)$ such that $S\cdot \Gamma\neq 0$. But since $f$ is projective, again by \cite[Theorem 12.1.3]{KM92}, $H_2(X/Y, \mbR)$ is generated projective curves. Hence, from the commutative diagram \eqref{eqn:final-diagram} it follows that $\alpha\in\Im(f^*)$ if and only if $\alpha\cdot C=0$ for all curves contained in the fibers of $f$. This completes our proof.  
\end{proof}

\begin{remark}
Assume that we are in Case (1) of the above Lemma \ref{lem:trivial-pullback} (i.e. $f$ is bimeromorphic). Let $T$ and $T'$ be two real closed bi-degree (1,1) currents on $X$ with local potentials such that $f_*T$ and $f_*T'$ also have local potentials. If $T\num T'$ in $H^2(X, \mbR)$, then from the above lemma it follows that $f_*T\num f_*T'$ in $H^2(Y, \mbR)$. We will use this fact throughout the article without reference. 
\end{remark}

%\bigskip\noindent{\bf Negativity Lemma} 
\subsection{Negativity Lemma} The following result is analogous to a well known negativity lemma for algebraic varieties (see eg.  \cite[Lemma 3.3]{Bir12}).
We include a proof, since we could not find a reference in the literature (see however \cite[Lemma 1.3]{Wang21} for the bimeromorphic case).
\begin{lemma}\label{l-1}
Let $f:X\to Y$ be a proper morphism of normal analytic varieties, where $X$ is a K\"ahler space, $f_*\OO _X=\OO _Y$ and $Y$ is relatively compact. Let $E=\sum a_iE_i\leq 0$ be a $f$-exceptional $\mbR$-Cartier divisor such that $-E$ is $f$-nef. Then $E=0$. 
\end{lemma}
Recall that by definition $E$ is $f$-exceptional if it is $f$-vertical and for any prime divisor $Q\subset f_*{\rm Supp}(E)$, there is a prime divisor $P$ on $X$ such that $P$ dominates $Q$ and $P\not \subset {\rm Supp}(E)$.
\begin{proof}
Suppose that $f({\rm Supp}(E))$ contains a divisor $Q$ and let $P,P'$ be divisors on $X$ dominating $Q$ such that $P\subset {\rm Supp}(E)$ and $P'\not \subset {\rm Supp}(E)$. Let $F$ be the fiber over a general point  $q\in Q$, then $Q$ is Cartier on a neighborhood of $q$ and we may assume that $P\cap P'\cap F \ne \emptyset $. 
Working over a neighborhood of $q\in Y$, we define \[\lambda:={\rm min}\{ t\geq 0\;|\;E+tf^*Q\geq 0\}.\] 
Then by our hypotheses $\lambda>0$ and $E+\lambda f^*Q\geq 0$. Possibly rechoosing $P,P'$, we may further assume that ${\rm mult}_P(E+\lambda f^*Q)=0$ whilst ${\rm mult}_{P'}(E+\lambda f^*Q)>0$.
Let $\bar F=F\cap P$ and $\omega$ be a K\"ahler form. Note that $\bar F$ is a compact analytic variety and $\omega|_{\bar F}$ is a K\"ahler form on $\bar F$. Let $d:=\dim \bar F$. Since $-E$ is $f$-nef and $0\leq (E+\lambda f^*Q)|_{\bar F}\ne 0$, we have 
\[0\leq (\omega |_{\bar F})^{d-1}\cdot (-E)=(\omega |_{\bar F})^{d-1}\cdot (-E-\lambda f^*Q)=(\omega |_{\bar F})^{d-1}\cdot(-E-\lambda f^*Q)|_{\bar F}<0,\] %\cdot \sum p_i\int _{E_i\cap F'} (\omega |_{F'})^{f-1}\leq p_1 \int _{E_1\cap F'} (\omega |_{F'})^{f-1}<0 \] 
which is a contradiction. 

We may therefore assume that ${\rm codim}(f({\rm Supp}(E)))\geq 2$.
Let $q$ be a general point on a maximal dimensional component of $f({\rm Supp}(E))$. Working locally over a neighborhood of $q\in Y$, we may assume that $Y$ is Stein.
Suppose that $\dim Y=2$ and let 
$q\in H\subset Y$ be a general hyperplane, then $f^*H=\sum h_iE_i+H'$ where $f(E_i)=q$ and each component of $H'$ maps to a curve on $Y$. Note that the the support of $\sum h_iE_i$ contains all divisors mapping to $q$ and we let 
\[\lambda :={\rm min}\{t\geq 0\;|\; E+t f^*H\geq 0\}.\]
%{\rm mult}_{S}(E+t p^*H)\geq 0\ \forall\ {\rm divs}\ S\ {\rm s.t.}\ p(S)=q\}.\]
 Then $\lambda >0$, $E+\lambda f^*H\geq 0$ and we may assume that ${\rm mult }_P(E+\lambda  f^*H)=0$ for some $p$-exceptional prime divisor $P\subset X$, whilst $P\cap {\rm Supp}(E+\lambda  f^*H)\ne \emptyset$ i.e. $(E+\lambda  f^*H)|_P$ is a non-zero effective divisor. Thus if $\dim X=n$, then we have 
 \[0\leq -E\cdot (\omega |_P)^{n-2}= (-E-\lambda f^*H)\cdot (\omega |_P)^{n-2}=-(E+\lambda  f^*H)|_P\cdot (\omega |_P)^{n-2}<0.\] This is a contradiction. 

If $\dim (f({\rm Supp}(E)))>0$, then pick a general point $q$ on an irreducible component of $f({\rm Supp}(E))$. Shrinking $Y$ is a neighborhood of $q$, we may assume that $Y$ is Stein and $f({\rm Supp}(E))$ is irreducible. Cutting by general hyperplanes $q\in H\subset Y$, we obtain $f':X'\to Y'$, $\dim Y'= 2$, $E'=E\cap X'\geq 0$ such that $-E'$ is $\mbR$-Cartier and $f'$-nef and $\dim (f({\rm Supp}(E')))=0$. Possibly replacing $X',Y'$ by their normalizations and shrinking $Y'$, we may assume that $X',Y'$ are normal and $f({\rm Supp}(E'))$ is a point on a surface and we conclude as above. 
\end{proof} 
\bigskip

\subsection{Graf-Kirschner's Decomposition of $H^2(X, \mbR)$}\label{subs:GK-decomposition} (See \cite[Section 4]{GK20}.) Let $X$ be a compact complex space of pure dimension and $H_*(X, \mbZ ):=H^{\rm BM}_*(X)$, the Borel-Moore homology. For any non-negative integer $k$, let $B_{2k}(X, \mbR)\subset H_{2k}(X, \mbR)$ be the real linear subspace spanned by the set of all $[A]\in H_{2k}(X, \mbR)$, where $A\subset X$ is a (complex) k-dimensional irreducible closed analytic subset of $X$. Let $N^1(X)_{\mbR}\subset H^2(X,\mbR)$ be the usual N\'eron-Severi group of $X$ given by the image of $H^1(X, \mcO_X^*)\to H^2(X, \mbR)$. Now we define $T(X)\subset H^2(X, \mbR)$ to be the subspace orthogonal to $B_2(X, \mbR)$ with respect to the pairing $H^2(X, \mbR)\times H_2(X, \mbR)\to \mbR$, i.e.
\begin{equation}\label{eqn:trivial-part}
T(X):=\{\alpha\in H^2(X, \mbR)\;|\; \langle \alpha, \beta\rangle=0 \mbox{ for all } \beta\in B_2(X, \mbR)\}. 
\end{equation}
Let $N_1(X)_{\mbR}$ be the space of curve classes up to numerical equivalence. If $X$ is a projective variety with rational singularities, then by \cite[Corollary 12.1.5.2]{KM92}, $B_2(X, \mbR)=N_1(X)_{\mbR}$. 

The following lemmas will be very useful in the rest of the article.
\begin{lemma}\cite[Proposition 4.2]{GK20}\label{lem:h2-decomposition}
    Let $X$ be a projective variety  with rational singularities. Then
    \[
    H^2(X, \mbR)=N^1(X)_{\mbR}\oplus T(X).
    \]
\end{lemma}

\begin{lemma}\label{lem:trans-alg}
 Let $X$ be a projective variety with rational singularities and $\alpha\in H^2(X,\mbR)$. Let $\alpha=D+\delta$ be the decomposition of $\alpha$ as above, where $D$ is an $\mbR$-Cartier divisor. Then the following hold
 \begin{enumerate}
     \item If $\alpha$ is K\"ahler, then $D$ is ample.
      \item If $\alpha$ is nef, then so is $D$.
     \item If $\alpha$ is big (resp. pseudo-effective), then so is $D$.
 \end{enumerate}    
\end{lemma}

    \begin{proof}
       (1) This follows from \cite[Proposition 4.5]{GK20}.\\

       (2) This is immediate from the definition of the decomposition of $\alpha$.\\

       (3) It is enough to prove this on a resolution of $X$. Thus, we may assume that $X$ is smooth here. If $D$ is not pseudo-effective, then by \cite[Theorem 2.2]{BDPP13} there is a strongly movable curve $\Gamma$ on $X$ such that $D\cdot \Gamma< 0$. But then $\alpha \cdot \Gamma=D\cdot \Gamma <0$ and
by \cite[Theorem A]{Nys19} $\alpha$ is not pseudo-effective.
Finally, suppose that $D$ is not big. If $A$ is ample, then $D-\epsilon A$ is not pseudo-effective for any $\epsilon >0$ and so, by what we showed above,  $\alpha-\epsilon A$ is not pseudo-effective for any $\epsilon >0$. But then $\alpha $ is not big.
       
    \end{proof}

%\begin{corollary}   Let $X$ be a normal projective variety and $\alpha \in H^{1,1}_{\rm BC}(X)$. If $\alpha$ is not pseudo-effective, then $\alpha \cdot C<0$ for some movable curve $C$.\end{corollary}
%\begin{proof}    Passing to a resolution, we may assume that $X$ is smooth. We write $\alpha=D+\delta$ where $D$ is an $\R$-Cartier divisor and $\delta \in T(X)$. By Lemma \ref{lem:trans-alg}\end{proof}

% \bigskip

\subsection{Generalized pairs.}\label{subs:g-pair}
Recall that a generalized pair $(X,B+\bbeta )$ consists of a compact normal variety $X$, a proper bimeromorphic map $\nu :X'\to X$ such that $X'$ is smooth, $B'$ is an $\R$-divisor with simple normal crossings on $X'$, $\beta '$ is a closed (1,1) current on $X'$ such that $[\beta']\in H^{1,1}_{\BC}(X')$ is a nef class and the class $[K_{X'}+B'+\beta ']\in H^{1,1}_{\BC}(X')$ is pulled back from $X$, i.e. $K_{X'}+B'+\beta '\equiv \nu ^*\gamma$ for some $\gamma \in H^{1,1}_{\rm BC}(X)$.
It is convenient to denote by $\bbeta=\overline {\beta '}$ the corresponding b-(1,1) current. This means that if $X'\dasharrow X''$ is any bimeromorphic map, then we define the trace $\bbeta _{X''}\in H^{1,1}_{\rm BC}(X'')$
as follows.
Let $p:W\to X'$ and $q:W\to X''$ be a common resolution, then $\bbeta _{X''}=q_*(p^* \beta ' )$. %We will say that $K_{X}+B+\bbeta$ is big (resp. $B+\bbeta$ is big) if $K_{X'}+B'+\bbeta_{X'}$ is big (resp. $(B')^{\geq 0}+\bbeta_{X'}$ is big).
We refer the reader to \cite{DHY23} for further details and properties of generalized pairs.
In particular we refer the reader to \cite[Definition 2.11]{DHY23} for the definition of log minimal
model for generalized pairs, and we recall that $(X,B+\bbeta)$ is a good minimal model if $K_X+B+\bbeta _X$ is semiample, i.e. there is a holomorphic map $f:X\to Z$ and a smooth K\"ahler form $\omega_Z$ on $Z$ such that $K_X+B+\bbeta _X\equiv f^* \omega_Z$.

\section{Proof of the main results}
In this section we will prove our main theorems, however, first we will 

\begin{lemma}\label{l-big}
    Let $(X,B+\bbeta)$ be a compact K\"ahler $\mbQ$-factorial generalized klt pair such that $B+\bbeta _X$ is big. Then there is a generalized klt pair $(X,G+\ggamma )$ and a K\"ahler form $\omega$ such that
    $K_X+B+\bbeta _X\equiv K_X+G+\ggamma _X+\omega$.
    
\end{lemma}
\begin{proof}
    Since $B+\bbeta _{X}$ is big, by Demailly's regularization theorem there is a K\"ahler current $T$ with weakly analytic singularities such that $B+\bbeta _{X}\equiv T$ (see \cite[Definition 4.11]{HP24}). By \cite[Lemma 3.7]{DH23} we may assume that $T\geq \omega_0$ and
    $\nu ^*T=\nu ^*\omega _0+\Theta +F$, where $\nu:X'\to X$ is a log resolution, $\omega _0$ is a K\"ahler form, $F$ is an effective $\mbR$ divisor with $\nu ({\rm Supp}(F))=E_+(T)$ and $\Theta$ is a closed positive $(1,1)$ current such that $[\Theta]\in H^{1,1}_{\BC}(X')$ is nef. Note that since $X$ is $\mbQ$-factorial, $\bbeta_X$ has local potentials, and thus by the negativity lemma  there is a $\nu$-exceptional effective divisor $E\geq 0$ such that $\bbeta_{X'}=\nu^*\bbeta_X+E$.
    %passing to a log resolution $\nu :X'\to X$, by \cite[Theorem 1.4]{Bou02} \hl{we may write that $\nu ^*(B+\bbeta _{X})\equiv \nu ^*\omega_0 +C'$ where $C'$ is an effective $\mbR$-divisor on $X'$ and $\omega _0$ is a K\"ahler form on $X$}.\footnote{Om: This is not true, \cite[Theorem 1.4]{Bou02} is the Fujita's approximation theorem for big $(1,1)$ class, it will give us $\nu^*(B+\bbeta_X)=\omega'+C'$, where $\omega'$ is a K\"ahler class on $X'$. I think the best we can do is prove a version of Lemma \ref{l-gklt2}(1), hopefully this is enough for applications?} 
    Therefore, for any $0<\epsilon \ll 1$ we have
    \begin{align*}
     \nu^*(K_X+B+\bbeta_X) &=K_{X'}+B'+\bbeta _{X'}\\
                            & \equiv K_{X'}+B'+\epsilon (E+F-\nu ^*B)+(1-\epsilon)\bbeta _{X'}+\epsilon \Theta + \nu ^* \omega,
    \end{align*}
     where $\omega:=\epsilon \omega _0$.
    We let $G':=B'+\epsilon (E+F-\nu ^*B)$, $G:=(1-\epsilon) B+\epsilon \nu _*F$, and $\ggamma:=(1-\epsilon)\bbeta+\epsilon \bar\Theta $.
    Since $K_{X'}+G'+\ggamma _{X'} \equiv \nu ^*(K_X+B+\bbeta _X- \omega)$, then $(X,G+\ggamma )$ is a generalized klt pair and $K_X+B+\bbeta _X\equiv K_X+G+\ggamma_X+ \omega $.
\end{proof}

\begin{lemma}\label{l-gklt}   Let $(X,B+\bbeta)$ be a compact K\"ahler generalized klt pair such that $K_X+B+\bbeta _X$ is  nef and big. Then for any $0<\epsilon \ll 1$ we can write
\[(1+\eps) (K_X+B+\bbeta_X) \equiv K_X+\Delta +\ggamma  _X+\omega \]   
such that $(X,\Delta +\ggamma )$ is a generalized klt pair, $\omega$ is a K\"ahler form, and in particular, $\Delta+\ggamma _X+\omega$ is  big.
\end{lemma}

\begin{proof}   
Let $\nu :X'\to X$ be a log resolution such that $\bbeta$ descends to $X'$ and
\begin{equation}\label{eqn:log}
    K_{X'}+B'+\bbeta_{X'}=\nu^*(K_X+B+\bbeta_X). 
\end{equation}
Since $K_X+B+\bbeta_X$ is nef and big, by \cite[Corollary 4.20]{HP24} there is a K\"ahler current $T$ with weakly analytic singularities such that $E^{as}_{nK}([K_X+B+\bbeta_X])=E_+(T)$. Let $\omega$ be a K\"ahler form such that $T\geq \omega$. Then passing to a higher resolution it follows from \cite[Lemma 3.7]{DH23} that 
\begin{equation}\label{eqn:resolution-of-current}
    \nu^*(K_X+B+\bbeta_X)\num F+\theta+\nu^*\omega,
\end{equation}
where $F$ is an effective $\mbR$-divisor such that $\nu(\Supp F)=E_+(T)$ and the current $\theta$ represents a nef class. Thus from \eqref{eqn:log} we have
 \[(1+\eps)(K_{X'}+B'+\bbeta_{X'})\num K_{X'}+B'+\bbeta _{X'}+\eps(F+\theta +\nu ^*\omega).\]
 Now let $\Delta':=B'+\eps F$ for $0<\eps\ll 1$, then $(X', \Delta')$ is sub-klt and $[\bbeta _{X'}+\eps \theta]$ is nef. Replacing $\bbeta_{X'}$ by a positive current in its class $[\bbeta_{X'}]\in H^{1,1}_{\BC}(X')$ we may assume that $\bbeta_{X'}+\eps\theta$ is a positive current. Thus $(X,\Delta+\ggamma_X)$ is generalized klt, where $\Delta:=\nu_*\Delta'$ and $\ggamma =\overline{\bbeta_{X'}+\eps \theta}$. Finally, replacing $\epsilon\omega$ by $\omega$ we have 
 \[(1+\eps)(K_X+B+\bbeta_X)\num K_X+\Delta+\ggamma _X+\omega\]
 such that $\Delta+\ggamma_X+\omega$ is big.

 %Moreover, since $X$ is $\mbQ$-factorial,\footnote{CH: We do not need $\Q$-factorial here. $\Delta+\ggamma _X+\eps\omega$ can be big without being locally exact (just as a Weil divisor can be big without being $\R$-Cartier). Maybe we remove the $\Q$-factorial hypotheses and say that $\Delta+\ggamma _X+\omega$ is the pushforward of a big (locally exact) class.} from \eqref{eqn:resolution-of-current} it follows that $\nu_*\theta\num (K_X+B+\bbeta_X)-(\nu_*F+\omega)$ has local potentials, and hence the class $[\ggamma _X]=[\bbeta_{X}+\eps\nu_*\theta]\in H^{1,1}_{\BC}(X)$ is pseudo-effective. Therefore, $\Delta+\ggamma _X+\eps\omega$ represents a big class, as $\omega$ is K\"ahler. Finally, replacing $\eps\omega$ by $\omega$ we have 

\end{proof}

%\begin{lemma}\label{l-gklt}    Suppose that $(X,B+\bbeta)$ is a klt generalized K\"ahler pair such that $\alpha =K_X+B+\bbeta _X$ is nef and big, then     $(1+\eps) \alpha =K_X+B'+\bbeta' _X+\omega '$    were $(X,B'+\bbeta )$ is a klt generalized pair and $\omega'$ is K\"ahler.\end{lemma}
%\begin{proof}    Let $\nu :X'\to X$ be a log resolution such that $\bbeta$ descends to $X'$ and $f^*\alpha =F+\theta +\eps f^*\omega$ where $\omega$ is K\"ahler on $X$ see \cite[Lemma 3.7]{DH23}.    We have \[(1+\eps)\nu^*\alpha =K_{X'}+B_{X'}+\bbeta _{X'}+\eps(F+\theta +\eps \nu ^*\omega)\] where $(X',B_{X'}+\eps F)$ is sub-klt, $\bbeta _{X'}+\eps \theta$ is nef, thus  $\alpha'=(1+\eps)\alpha$ is nef and big, $(X,B'+\bbeta ')$ is generalized klt where $B'=B+\epsilon \nu _* F$ and $\bbeta '=\bbeta +\eps \bar \theta$ and $\alpha ' =K_X+B'+\bbeta' _X+\omega '$ where $\omega'=\eps \omega $ is K\"ahler.\end{proof}
\begin{corollary}\label{c-nk}    
Let $(X,B+\bbeta)$ be a generalized klt pair, where $X$ is a projective variety over $\mbC$. \begin{enumerate}
    \item If the class of $\alpha=K_X+B+\bbeta _X$ is  nef and big but not K\"ahler, then there is a rational curve $C\subset X$ such that $\alpha \cdot C=0$.
    \item If the class of $\alpha=K_X+B+\bbeta _X$ is pseudo-effective but not nef, then there is a rational curve $C\subset X$ such that $\alpha \cdot C<0$.
\end{enumerate}
\end{corollary}
\begin{proof} 
(1) By Lemma \ref{l-gklt}, we may assume that $\alpha\num K_X+\Delta+\ggamma  _X+\omega$ such that $(X,\Delta+\ggamma  )$ is gklt, $\Delta +\ggamma  _X+\omega$ is big, and $\omega $ is K\"ahler. Since $K_X+\Delta +\ggamma  _X+\omega$ is big, then so is $K_X+\Delta +\ggamma  _X+(1-\epsilon)\omega$  for $0<\epsilon \ll 1$. Then we are done by \cite[Proposition 3.1]{HP24}.

(2) Let $\omega$ be a K\"ahler form and $\lambda>0$ be the nef threshold so that $K_X+B+\bbeta _X+\lambda \omega$ is nef but not K\"ahler. Since $\alpha=K_X+B+\bbeta _X$ is pseudo-effective, then $K_X+B+\bbeta _X+\lambda \omega$ is big. By part (1), there is a rational curve $C\subset X$ such that $(K_X+B+\bbeta _X+\lambda \omega) \cdot C=0$. But then $\alpha\cdot C=(K_X+B+\bbeta _X) \cdot C=-(\lambda \omega)\cdot C <0$.
\end{proof}

\begin{lemma}\label{l-gklt2}
    Let $(X,B+\bbeta )$ be a generalized klt pair such that $X$ is a compact K\"ahler variety, $K_X$ is $\mbQ$-Cartier, and $B+\bbeta_X$ is big. %Assume that $X$ is a $\mbQ$-factorial projective variety over $\mbC$. Then we can write
    Then the following hold:
    \begin{enumerate}
      \item If $X$ is $\Q$-factorial, then there exists a generalized klt pair $(X,G+\ggamma)$ such that
        $\ggamma_{X'}$ represents a K\"ahler class on a log resolution $X'\to X$ and $K_X+B+\bbeta _X\equiv K_X+G+\ggamma _X$.
        \item If $X$ is projective, then we can write $K_X+B+\bbeta _X\equiv K_X+\Delta+\delta$, where $(X,\Delta )$ is a klt pair, $\Delta $ is a big $\mbR$-Cartier divisor and $\delta\in T(X)$ (see \eqref{eqn:trivial-part}).
        \end{enumerate}
\end{lemma}
\begin{proof}
    
   (1) %Since $B+\bbeta _{X}$ is big, passing to a log resolution $\nu :X'\to X$, by \cite[Theorem 1.4]{Bou02} we may write that $\nu ^*(B+\bbeta _{X})\equiv  \omega ' +C'$ where $C'$ is an effective $\mbR$-divisor on $X'$ and $\omega '$ is a K\"ahler form on $X'$. As usual we let $K_{X'}+B'+\bbeta _{X'}=\nu ^*(K_X+B+\bbeta _X)$. For any $0<\epsilon \ll 1$, we let \[K_{X'}+B^\epsilon +\beta ^\epsilon =K_{X'}+B'+\bbeta _{X'}+\epsilon (\omega '+C'-\nu ^*(B+\bbeta _X))\]
   By Lemma \ref{l-big} we may assume that there is a generalized klt pair $(X,G+\ggamma  )$ and a K\"ahler form $\omega$ such that
    $K_X+B+\bbeta _X\equiv K_X+G+\ggamma _X+\omega$. Let $\nu:X'\to X$ be a log resolution of $(X, B+\bbeta)$  and $(X,G+\ggamma  )$ such that there is an effective $\nu$-exceptional $\mbQ$-divisor $E'$ with $-E'$ is $\nu$-ample.
    Write $K_{X'}+G'+\ggamma_{X'}=\nu^*(K_X+G+\ggamma_X)$. For any $0<\epsilon \ll 1$, $\ggamma _{X'}+\nu^*\omega -\epsilon E'$ represents a K\"ahler class in $H^{1,1}_{\rm BC}(X')$ and $(X', G'+\epsilon E')$ is sub klt. Let $\eta'$ be a K\"ahler form cohomologous to $\ggamma _{X'}+\nu^*\omega -\epsilon E'$. Then replacing $\ggamma$ by $\overline{\eta'}$,  (1) follows.

   (2) Let $\nu :X'\to X$ be a log resolution of $(X, B+\bbeta)$ such that $K_{X'}+B'+\bbeta _{X'}=\nu ^*(K_X+B+\bbeta _X)$,  where $[\bbeta _{X'}]\in H^{1,1}_{\BC}(X')$ is nef. Then by Lemmas \ref{lem:h2-decomposition} and \ref{lem:trans-alg}, there is a unique decomposition $\bbeta_{X'}\equiv N_{X'}+\delta '$ such that $N_{X'}$ is a nef $\mbR$-Cartier divisor and $\delta'\in T(X')$. Thus $(X,B+\mathbf N)$ is a generalized pair in the usual sense (see \cite{BZ16}), where $\mathbf N:=\overline {N_{X'}}$. 
 We have \[K_X+B+\mathbf N_X+\nu_*\delta'=\nu_*(K_{X'}+B'+\mathbf N_{X'}+\delta') \equiv \]\[\nu_*(K_{X'}+B'+\bbeta_{X'})\equiv K_X+B+\bbeta_X,\] and $\delta'=\nu^*\delta$ for some $\delta \in T(X)$ 
 by Lemma \ref{lem:trivial-pullback}. Thus it follows from Lemmas \ref{lem:h2-decomposition} and \ref{lem:trans-alg} that $B+\mathbf N _X$ is an $\mbR$-Cartier big divisor. Now since $\nu_*(B'+\mathbf N_{X'})=B+\mathbf N_X$, there is a $\nu$-exceptional effective $\mbR$-divisor $F'\geq 0$ such that $B'+\mathbf N_{X'}+F'\geq \nu^*(B+\mathbf N_X)$, and hence $B'+\mathbf N_{X'}+F'$ is big. Let $B'+F'+\mathbf N_{X'}\sim_{\mbR} A'+E'$, where $A'$ is an ample $\mbQ$-divisor and $E'$ is an effective $\mbR$-divisor. Then for $0<\epsilon\ll 1$, $(X', B'+\epsilon(-F'-B'+E'))$ is sub-klt. Pick a general effective $\mbR$-divisor $0\leq G'\sim_{\mbR} \epsilon A'+(1-\epsilon)\mathbf N_{X'}$ such that $(X', \Delta':=B'+\epsilon(-F'-B'+E')+G')$ is sub-klt. Now observe that \[K_{X'}+\Delta'+\delta'\equiv K_{X'}+B'+\mathbf N_{X'}+\delta'\equiv K_{X'}+B'+\bbeta_{X'}=\nu^*(K_X+B+\bbeta_X).\] 
 By Lemma \ref{lem:trivial-pullback}, it follows that $K_{X'}+\Delta'\equiv \nu^*(K_X+\Delta)$, where $\Delta:=\nu_*\Delta'$. Thus $(X, \Delta)$ is klt and $\Delta$ is big, as $\Delta+\nu_*\delta'\equiv B+\bbeta_X$ is big. Clearly we  have $K_X+B+\bbeta_X\equiv K_X+\Delta+\delta$, where $\delta:=\nu_*\delta'\in T(X)$. This completes the proof of (2).

\end{proof}

We will now prove our main theorem.
\begin{proof}[Proof of Theorem \ref{t-main}]
    By Lemma \ref{l-gklt2}(2) we can write $K_X+B+\bbeta _X\equiv K_X+\Delta+\delta$, where $(X,\Delta )$ is klt, $\Delta $ is big and $\delta\in T(X)$. We run the $(K_X+\Delta)$-MMP with scaling of an ample divisor.
    This MMP terminates by \cite[Corollary 1.4.2]{BCHM10}; let $\phi :X\dasharrow X'$ be the output of this MMP. Since $\delta\in T(X)$, each step of this MMP is $\delta$-trivial and $\delta':=\phi_*\delta\in T(X')$ by Lemma \ref{lem:trivial-pullback}(1). 
    
    If %$K_X+B+\bbeta$ is pseudo-effective then so is 
    $K_X+\Delta$ is pseudo-effective, then it follows that   $K_{X'}+\Delta '$ is nef and $\Delta '$  is big, hence $K_{X'}+\Delta '$ is semi-ample (see \cite[Corollary 3.9.2]{BCHM10}). Let $g:X'\to Z$ be the contraction induced by $K_{X'}+\Delta'$. Since $\Delta'$ is big, by a standard perturbation technique we can write $\Delta'\sim_{\mbR} \Delta''+H''$, where $H''$ is an ample $\mbR$-Cartier divisor, and $\Delta''\>0$ is a $\mbQ$-divisor such that $(X',\Delta'')$ is klt. Then $-(K_{X'}+\Delta'')\num_Z H''$ is $g$-ample, and thus by \cite[Lemma 8.8]{DHP22}, $Z$ has rational singularities. Therefore by Lemma \ref{lem:trivial-pullback}(2), there is $\delta_Z\in T(Z)$ such that $\delta'=g^*\delta _Z$.

    Now, by the canonical bundle formula as in \cite[Theorem 3.1]{FG12} and \cite[Theorem 4.1]{Amb05}, $K_{X'}+\Delta '\equiv g^*(K_Z+\Delta _Z)$, where $(Z,\Delta _Z)$ is klt and $A_Z:=K_Z+\Delta _Z$ is ample.
    Note that $K_{X'}+B'+\bbeta _{X'}\equiv K_{X'}+\Delta'+\delta'\equiv g^*(A_Z+\delta_Z)$ and so by the canonical bundle formula of \cite[Theorem 2.3]{HP24}, $K_{X'}+B'+\bbeta _{X'}\equiv g^*(K_Z+B_Z+\bbeta ^Z_Z)$, where $(Z,B_Z+\bbeta ^Z)$ is generalized klt.
    
     We will now prove that $\alpha _Z$ is K\"ahler. Suppose that $\alpha _Z:=K_Z+B_Z+\bbeta ^Z_Z$ is big but not K\"ahler, then by Corollary \ref{c-nk}, there is a rational curve $C\subset Z$ such that $\alpha _Z\cdot C\leq 0$ and hence \[ A_Z\cdot C=(K_{Z}+\Delta _Z)\cdot C= (K_{Z}+\Delta _Z+\delta _Z)\cdot C=\alpha _Z \cdot C\leq 0\] which is impossible. Thus $\alpha _Z$ is K\"ahler in this case.

    Suppose now that $\alpha _Z:=K_Z+B_Z+\bbeta ^Z_Z$ is not big. Since $B'+\bbeta_{X'}$ is big,   by Lemma \ref{l-big} we can write $B'+\bbeta_{X'}\equiv B^* +\bbeta_{X'} ^\sharp +\omega'$ such that $(X, B^* +\bbeta_{X'} ^\sharp)$ is gklt and $\omega '$ is a  K\"ahler form. Let $H_Z$ be a general ample $\mbQ$-divisor on $Z$ such that $\omega'':=\omega'-g^*H_Z$ represents a K\"ahler class. Set $\bbeta_{X'}^*:=\bbeta_{X'}^\sharp+\overline{\omega''}$, then  $B'+\bbeta_{X'}\equiv B^*+\bbeta_{X'} ^*+g^*H_Z$ such that $(X', B^*+g^*H_Z+\bbeta_{X'} ^*)$ is gklt. Thus $B_Z+\bbeta ^Z_Z\equiv B_Z^*+\bbeta ^{*,Z}_Z+H_Z$ and so $B_Z+\bbeta ^Z_Z$ is big. Now if $\mu:\bar Z\to Z$ is a small $\mbQ$-factorization of $Z$, then $K_{\bar Z}+B_{\bar Z}+\bbeta^Z_{\bar Z}=\mu^*(K_Z+B_Z+\bbeta^Z_Z)$, and hence $K_{\bar Z}$ is not pseudo-effective, as $B_{\bar Z}+\bbeta^Z_{\bar Z}$ is big but $K_{\bar Z}+B_{\bar Z}+\bbeta^Z_{\bar Z}$ is not. In particular, $\bar Z$, and hence $Z$ is uniruled.
     % Therefore $K_Z$ is not pseudo-effective (as $K_Z+B_Z+\bbeta ^Z_Z$ is not big), \hl{and hence $Z$ is uniruled.}\footnote{Om: Does this make sense for $K_Z$ non $\mbQ$-Cartier? CH: Yes, however, if this makes you nervous just take a small $\Q$-factorialization $Z^q\to Z$; by the same argument $K_{Z^q}$ is not PSEF.} 
    Let $h:Z\dasharrow W$ be the maximal rationally chain
connected (MRCC) fibration. Then $h$ is a proper morphism over an open subset $U\subset W$. Let $F$ be a general fiber of $h$. Then by adjunction $(F, \Delta_F)$ is klt, where $(K_Z+\Delta_Z)|_F=K_F+\Delta_F$. Thus $F$ has rational singularities, and it is also rationally connected due to \cite[Corollary 1.5(2)]{HM07}. If $\tilde F\to F$ is a resolution of singularities of $F$, then by \cite[Corollary 4.18(a)]{Deb01}, $H^2(F, \mcO_F)\cong H^2(\tilde F, \mcO_{\tilde F})\cong \overline{H^0(\tilde F, \Omega_{\tilde F}^2)}=0$. In particular, $\delta _Z|_F=0$ and so $\alpha _Z|_F\equiv A_Z|_F$ is ample.

Following \cite{Gue20} and \cite[Theorem 5.2]{CH20}, we will show that then $\alpha _Z$ is big, which is impossible. To see this, let $\nu :Z'\to Z$ be a resolution of singularities of $Z$ such that the composite $h':=h\circ \nu: Z'\to W$ is a morphism. We may assume that $W$ is smooth and there is an effective $\nu$-exceptional $\mbR$-divisor $E\geq 0$ on $Z'$ such that $-E$ is $\nu$-ample. Let $\alpha_{Z'}:=\nu ^*\alpha_Z$ and $\alpha _\epsilon:=\alpha _{Z'}-\epsilon E$ for $\epsilon>0$. If $F'$ is a general fiber of $h'$, then $\alpha _\epsilon|_{F'}$ is a K\"ahler class for any $0<\epsilon\ll 1$, and thus $(K_{Z'}+t\alpha _\epsilon )|_{F'}$ is K\"ahler for $t\gg 0$. Then by \cite[Theorem 5.2]{CH20}, $K_{Z'/W}+t\alpha _\epsilon $ is pseudo-effective. Since $W$ is projective and not uniruled, we know that $K_W$ is pseudo-effective, and hence so is  
\[K_Z+t\alpha_Z =\nu_*(K_{Z'}+t\alpha _\epsilon ).\]
Since $B_{Z}+\bbeta ^Z$ is big, then  \[\alpha _Z=\frac 1{t+1}(K_{Z}+B_{Z}+\bbeta ^Z_Z+t\alpha _Z)\]
is big. This is the required contradiction.\\

%Let $\nu :Z'\to Z$ be a resolution such that $g':Z'\to W$ is a morphism. We let $\Delta'$

% Note that in this case $K_{X'}+B'+\bbeta _{X'}$ is pseudo-effective and hence so is  $K_{X}+B+\bbeta _{X}$.\\ 

    If instead $K_X+\Delta$ is not pseudo-effective, then we have a Mori fiber space $g:X'\to Z$ such that $-(K_{X'}+\Delta ')$ is $g$-ample.
Thus by \cite[Lemma 8.8]{DHP22} and Lemma \ref{lem:trivial-pullback}(2),
$\delta'=g^*\delta _Z$ for some $\delta _Z\in T(Z)$. %In particular, $-(K_{X'}+B'+\bbeta_{X'})\equiv _Z -(K_{X'}+\Delta')$ is also $g$-ample. 
Note that then $\phi ^{-1}: X'\dasharrow X$ is an isomorphism on a big open subset $U\subset X'$ (i.e. the complement of a codimension 2 subset). The general fiber of $g$ is projective and a general complete intersection $C$ curve on such a fiber is contained in $U$. Thus 
\[(K_X+B+\bbeta _X)\cdot \phi ^{-1}_*C =(K_{X'}+B'+\bbeta _{X'})\cdot C=(K_{X'}+\Delta')\cdot C<0.\] 
Thus $K_X+B+\bbeta _X$ is not pseudo-effective.
\end{proof}

\begin{corollary}
    Let $(X,B+\bbeta)$ be a projective generalized klt pair with $\mbQ$-factorial singularities such that $K_X+B+\bbeta_X$ is big. Then $(X,B+\bbeta)$ admits a good minimal model.
\end{corollary}
\begin{proof}
Since $K_X+B+\bbeta_X$ is big, passing to a log resolution $\nu:X'\to X$ of $(X, B+\bbeta)$ and then using \cite[Theorem 1.4]{Bou02} we may write that \[K_{X'}+B'+\bbeta_{X'}\num \nu^*(K_X+B+\bbeta_X)=\eta'+C',\] where $\eta'$ is a K\"ahler class and $C'$ an effective $\mbR$-divisor.  %we may write $K_X+B+\bbeta_X\equiv C+\eta$, where $\eta$ is a modified K\"ahler class and $C$ is an effective $\mbR$-Cartier divisor (cf. \cite[Theorem 1.4]{Bou02}). 
% Let $\nu:X'\to X$ be a resolution such that $\bbeta$ descends to $X'$ and $\eta=\nu _*\eta '$ where $\eta'$ is K\"ahler and $\nu ^*\eta =\eta '+E$ for some effective $\nu$-exceptional $\R$-divisor $E$ (see \cite[Proposition 2.3]{Bou04}).
Then $(X',G':=B'+\epsilon C')$ is sub-klt and $\bbeta _{X'}+\epsilon \eta '$ is K\"ahler for any $0<\epsilon\ll 1$. Let $\ggamma :=\overline{\bbeta _{X'}+\epsilon \eta '}$ and $G:=\nu _*G'$, then $(X,G+\ggamma)$ is generalized klt, $G+\ggamma_X$ is big
and \[(1+\epsilon) (K_X+B+\bbeta_X) \equiv K_X+G+\ggamma _X .\]
  %  By Lemma \ref{l-gklt}, for any $0<\epsilon \ll 1$ we have         \[(1+\eps) (K_X+B+\bbeta_X) \equiv K_X+\Delta +\ggamma  _X+\omega, \]     where $(X,\Delta +\ggamma+\boldsymbol{\omega} )$ is a generalized klt pair such that $\Delta+\ggamma_X+\boldsymbol{\omega}_X$ is big, and $\boldsymbol{\omega}=\overline{\omega}$.
    By Theorem \ref{t-main}, $(X,G +\ggamma )$ has a good minimal model, which is also a good minimal model for  $(X, B+\bbeta)$.
\end{proof}

\begin{corollary}\label{c-tran} Let $(X, B + \bbeta )$ be a projective generalized klt pair such that $K_X$ is $\mbQ$-Cartier and $B + \bbeta_X$ is big. If $K_X+B+\bbeta_X$ is nef, then it is semiample.
\end{corollary}
\begin{proof}
 By Lemma \ref{l-gklt2}(2), there is a big $\mbR$-Cartier divisor $\Delta\geq 0$ such that $(X, \Delta)$ is klt. Thus $X$ admits a small $\mbQ$-factorization $\nu :X'\to X$.
We write $K_{X'}+B'+\bbeta _{X'}=\nu ^*(K_X+B+\bbeta_X)$. Then $(X',B'+\bbeta _{X'})$ is a $\mbQ$-factorial minimal model, and by Theorem \ref{t-main}, there is a contraction $g:X'\to Z$ and a K\"ahler class $\omega $ on $Z$ such that $K_{X'}+B'+\bbeta _{X'}=g^*\omega $.

Let $h:X\dasharrow Z$ be the induced meromorphic map. 
We claim that $h$ is a morphism. To see this, let $C\subset X'$ be a curve that is $\nu $-vertical but not $g$-vertical. Then
\[0=\nu ^* (K_X+B+\bbeta_X)\cdot C=(K_{X'}+B'+\bbeta _{X'})\cdot C=\omega\cdot g_* C>0.\]
This is impossible, so no such curves exist. By the rigidity lemma, it follows that $h$ is a morphism.
But then $K_X+B+\bbeta_X\equiv h^*\omega $, i.e. $K_X+B+\bbeta_X$ is semiample.
\end{proof}

As an application of this corollary we get the following version of the Conjecture \ref{c-tossati} for projective varieties.
\begin{corollary}\label{cor:tossati-conj}
    Let $(X, B)$ be a projective klt pair and $\alpha\in H^{1,1}_{\BC}(X)$ a nef class. If $\alpha-(K_X+B)$ is nef and big, then there is a projective surjective morphism $f:X\to Z$ with connected fibers to a normal projective variety $Z$ with rational singularities such that $\alpha=f^*\omega_Z$ for some smooth K\"ahler form $\omega_Z$ on $Z$. 
\end{corollary}

\begin{proof}
Let $\omega$ be a K\"ahler current on $X$ such that $\omega\equiv \alpha-(K_X+B)$ in $H^{1,1}_{\rm BC}(X)$. Let $\oomega:=\overline\omega$; then $(X, B+\oomega)$ is a gklt pair such that $B+\oomega_X=B+\omega$ is big.  Since $X$ is a projective variety and $(X, B)$ is klt, it is well known that $(X, B)$ has a small $\mbQ$-factorization, say $g:(X', B')\to (X, B)$ such that $(X', B')$ is a $\mbQ$-factorial klt pair and $K_{X'}+B'=g^*(K_X+B)$. Write $\alpha':=g^*\alpha$ and $\omega':=g^*\omega$. Then $\oomega_{X'}=\omega'$, $(X', B'+\oomega)$ is a gklt pair such that  $\alpha'=K_{X'}+B'+\oomega_{X'}$ is nef, and $B'+\oomega_{X'}$ is big. Thus by Corollary \ref{c-tran}, $\alpha'=K_{X'}+B'+\omega'=K_{X'}+B'+\oomega_{X'}$ is semiample, i.e. there is a proper surjective morphism $h:X'\to Z$ with connected fibers to a normal compact K\"ahler variety $Z$ and a K\"ahler class $\omega_Z\in H^{1,1}_{\BC}(Z)$ such that $K_{X'}+B'+\omega'\num h^*\omega_Z$. Then by a similar argument as in the proof of Corollary \ref{c-tran} using the rigidity lemma it follows that there is a unique morphism $f:X\to Z$ satisfying $f\circ g=h$; in particular, $\alpha=K_X+B+\omega\num f^*\omega_Z$. Moreover, from $K_{X'}+B'+\omega'\num h^*\omega_Z$ we have $-(K_{X'}+B')\equiv_h \omega'$ is nef and big, and thus by \cite[Lemma 8.8]{DHP22}, $Z$ has rational singularities. Now observe that $Z$ is Moishezon (as it is dominated by the projective variety $X'$), and also a K\"ahler space with rational singularities,  thus by \cite[Theorem 1.6]{Nam02}, $Z$ is projective.
\end{proof}

%\begin{corollary}\label{c-nk+}     Let $(X,B+\bbeta)$ be a generalized klt pair, where $X$ is a projective $\mbQ$-factorial variety over $\mbC$. If $\alpha =K_X+B+\bbeta _X$ is nef and  $B+\bbeta $ is  big, then there is a rational curve $C\subset X$ such that $\alpha \cdot C=0$.\end{corollary} \begin{proof} \end{proof}

\section{Applications}
In this section we give some applications of our main theorems. 
The following application was suggested to us by Mihai P\u{a}un. It is a generalization of \cite[Theorem 1.11]{Bir16}.
\begin{theorem}
    Let $(X,B+\bbeta)$ be a projective generalized pair such that $K_X$ is $\mbQ$-Cartier, $B+\bbeta_X$ is big, and $\alpha:=K_X+B+\bbeta _X$ is nef but not big. Then through every point $x\in X$ there is a rational curve $\Gamma_x$ with $\alpha \cdot \Gamma_x=0$. 
\end{theorem}
\begin{proof}
   Let $\nu :X'\to X$ be a dlt model for $(X,B +\bbeta)$ (see the proof of \cite[Theorem 1.6]{HP24}).
    Then $(X',B'+\bbeta )$ is a generalized dlt pair and $X$ is strongly $\mbQ$-factorial, where $B':=\nu ^{-1}_*B +{\rm Ex}(\nu)$ and $K_{X'}+B'+\bbeta _{X'}+D'=\nu ^*(K_X+B + \bbeta _X)$ for some effective $\R$-divisor $D '\geq 0$.
Replacing $(X, B+\bbeta)$ by $(X', B'+D'+\bbeta)$ we may assume that $X$ has strongly $\mbQ$-factorial klt singularities.
    
    Since $B+\bbeta _X$ is big and $\alpha\equiv K_{X}+B+\bbeta _{X}$ is not big, then $K_{X}$ is not pseudo-effective, and hence $X$ is uniruled. Let $g:X\dasharrow Z$ be the maximal rationally chain connected fibration, then $g$ is a proper morphism over an open subset of $Z$. Let $F$ be a general fiber of $g$, then $F$ is rationally chain connected. Since $X$ is klt, by adjunction $F$ has klt singularities. Then by \cite[Corollary 1.5(2)]{HM07}, $F$ is rationally connected. Moreover, since $F$ has rational singularities (as it is klt), by \cite[Corollary 4.18(a)]{Deb01}, $H^2(F, \mcO_F)=0$.  
   % We also note that since $X$ is klt, then $X$ has rational singularities and hence $F$ also has rational singularities.
In particular, $\alpha |_F$ is the class of a nef $\mbR$-divisor. 

    By Lemma \ref{lem:h2-decomposition} and \ref{lem:trans-alg}, we may write $\alpha =K_X+\Delta_0 +\delta$, where $\Delta_0$ is a big $\mbR$-divisor and $\delta \in T(X)$. Write $\Delta _0\equiv A+\Delta$, where  $A$ is ample $\mbQ$-divisor and $\Delta \geq 0$ is an effective $\mbR$-divisor. 

    We claim that $\alpha |_F$ is not big. To see this it suffices to show that if  $\alpha |_F$ is big, then $\alpha$ is big. Passing to  higher resolutions $\nu :X'\to X$ and $\mu:Z'\to Z$, we may assume that $g':X'\to Z'$ is a morphism of smooth projective varieties and we may write $(B')^{\geq 0}+\bbeta _{X'} \equiv \omega' +G'$, where $\omega'$ is K\"ahler and $G'$ is an effective $\R$-divisor. Note that since $\alpha |_F$ is big then so is $\alpha '|_{F'}$, where $\alpha':=\nu ^* \alpha$. It follows that $(K_{X'}+t\alpha ')|_{F'}$ is big for $t\gg 0$. Since $t\alpha '+\frac 12 \omega '$ is K\"ahler, by \cite[Theorem 5.2]{CH20}, it follows that $K_{X'/Z'}+t\alpha' +\frac 12 \omega'$ is pseudo-effective for $t\gg 0$. Since $K_{Z'}$  is pseudo-effective, it follows that $K_{X'}+t\alpha' +\frac 12 \omega'$ is pseudo-effective and  
    \[(1+t)\alpha'+(B')^{<0} \equiv K_{X'}+t\alpha' + \omega'+G'\] is big. 
    Since $(B')^{<0}$ is $\nu$-exceptional, then \[\alpha=\nu _*\left(\alpha'+\frac 1 {1+t}(B')^{<0}\right)\] is also big. This is the required contradiction.

    Let $d=\dim F$. Since $\alpha|_F$ is represented by a nef $\mbR$-Cartier divisor but not big, then by \cite[Lemma 12.1]{Bir16}, it follows that there exists a $0<\varepsilon \ll 1$ and a very ample divisor $H$ on $F$  such that $(\alpha -\varepsilon A)|_F\cdot H^{d-1}=0$. We remark that as $H|_F^{d-1}$ is the class of a curve on $F\subset X$, then $\delta \cdot H^{d-1}=0$. 
    It follows that if we fix $m>0$, then for  $t=\frac {1-\varepsilon}{1+m\varepsilon}$ we have   %$\delta =\frac{1-t}{1+tN}$ $\delta (1+tN)=1-t$ $t(N\delta +1)=1-\delta$ $t=\frac {1-\varepsilon}{1+N\varepsilon}$
    \[(K_X+\Delta +t(A+m\alpha))|_F\cdot H^{d-1}=(\alpha -A+t(A+m\alpha))|_F\cdot H^{d-1}=0\] where $0<1-t\ll 1$.
In particular, through a general point of $x\in F$ there is a complete intersection curve $\mathcal C _x\subset F$ such that  $(K_X+\Delta +t(A+m\alpha))\cdot \mathcal C _x=0$.
We also note a general complete intersection curve $\mathcal C _x$
is contained in the smooth locus of $F$ and since $A+\Delta$ is big, then $K_F\cdot \mathcal C_x=K_X\cdot \mathcal C_x<0$.
Arguing as in  the proof \cite[Theorem 1.11]{Bir16}, we then have a rational curve $\mathcal L_x$ passing through a general point $x\in X$ such that
\[0<A\cdot \mathcal L _x\leq (A+m\alpha)\cdot \mathcal L_x\leq 3\dim X.\] 
Thus these curves belong to a bounded family (independent of $m$), and hence they belong to finitely many distinct numerical equivalence classes. In particular, there is a rational curve $\Gamma_x\subset X$ of this family $\{\mathcal L_x\}$ such that $(A+m\alpha)\cdot \Gamma_x\leq 2\dim X$ for infinitely many $m\in\mbZ^+$. This is a contradiction unless $\alpha\cdot \Gamma_x=0$, as $A$ is ample and $\alpha$ is nef.

%It follows that $\alpha \cdot \Gamma_x=0$. 
\end{proof}

Next we prove the transcendental cone theorem for projective varieties. 
\begin{theorem}\label{thm:cone}
    Let $(X,B+\bbeta )$ be a projective generalized klt pair. Then there are at most countably many rational curves $\{\Gamma_i\}_{i\in I}$ such that $0<-(K_X+B+\bbeta _X)\cdot \Gamma _i\leq 2\dim X$ and 
    \[\overline{\rm NA}(X)=\overline{\rm NA}(X)_{(K_X+B+\bbeta _X)\geq 0}+\sum _{i\in I}\R ^+[\Gamma _i].\] 
Moreover,  if  $K_X$ is $\mbQ$-Cartier and  $B+\bbeta_X $ is big, then $I$ is finite.

\end{theorem}
\begin{proof} 
By standard arguments (see e.g. the proof of \cite[Corollary 5.3]{DHP22}), it is enough to show the following: (i) there is a strongly $\mbQ$-factorial model $f:X'\to X$ such that $K_{X'}+B'+\bbeta_{X'}=f^*(K_X+B+\bbeta_X)$, $B'\geq 0$ is an effective $\mbR$-divisor, and $(X', B'+\bbeta)$ is generalized klt, and (ii) the cone theorem holds for $(X', B'+\bbeta)$.

 First we will establish (i). Let $\nu:X'\to X$ be a log resolution of $(X, B+\bbeta)$ such that $K_{X'}+B'+\bbeta_{X'}=\nu^*(K_X+B+\bbeta_X)$. Let $\bbeta_{X'}\equiv N_{X'}+\delta'$ be the decomposition as in Lemma \ref{lem:h2-decomposition}. Then by Lemma \ref{lem:trans-alg}, $N_{X'}$ is a nef $\mbR$-Cartier divisor on $X'$. Let $\mathbf N:=\overline{N_{X'}}$ be the b-Cartier closure of $N_{X'}$. Now by Lemma \ref{lem:trivial-pullback} there is a $\delta\in T(X)$ such that $\delta'=\nu^*\delta$. Therefore \[K_X+B+\mathbf{N}_X\equiv \nu_*(K_{X'}+B'+\bbeta_{X'}-\nu^*\delta)=(K_X+B+\bbeta_X)-\delta\] is $\mbR$-Cartier. Moreover, we have $K_{X'}+B'+\mathbf N_{X'}=\nu^*(K_X+B+\mathbf N_X)$, and hence $(X, B+\mathbf{N})$ is a generalized klt pair in the usual sense (see \cite{BZ16}). 
 % To this end recall from the proof of Lemma \ref{l-gklt2}(2) that if $\nu:X'\to X$ is a log resolution of $(X, B+\bbeta)$ and $\bbeta_{X'}\num N_{X'}+\delta'$ as in Lemma \ref{lem:trans-alg}, then $(X, B+\mathbf{N})$ is a usual generalized klt pair, where $N_{X'}$ is a nef $\mbR$-divisor,  $\mathbf{N}=\overline{N_{X'}}$, and $\delta'\in T(X')$.
 Then by \cite[Theorem 2.9, 3.10]{FS23}, there is a strongly $\mbQ$-factorial model $f:X''\to X$ such that $K_{X''}+B''+\mathbf{N}_{X''}=f^*(K_X+B+\mathbf N_X)$, $B''\geq 0$ is an effective $\mbR$-divisor and $(X'', B''+\mathbf N)$ is a usual generalized klt pair. Since this model is obtained by running an MMP over $X$, we see that $\delta'':=\phi_*\delta'\in T(X'')$ and $\bbeta_{X''}\num \mathbf N_{X''}+\delta''$. In particular, $K_{X''}+B''+\bbeta_{X''}=f^*(K_X+B+\bbeta_X)$. Thus replacing $(X', B'+\bbeta)$ by $(X'', B''+\bbeta)$ and $\nu$ by $f$ we may assume that $(X', B'+\bbeta)$ is a strongly $\mbQ$-factorial generalized klt pair.

Now we will proceed to prove (ii). Arguing as in the proof of \cite[Theorem 1.6]{DHY23}, it suffices to show that if $\omega $ is a K\"ahler form such that $\alpha =K_X+B+\bbeta _X+\omega $ is a nef supporting class of an exposed extremal ray of $\NA(X)$, then there is a rational curve $C\subset X$ such that $0<-(K_X+B+\bbeta _X)\cdot C\leq 2\dim X$. Note that $(X, B+\bbeta+\oomega)$ is a gklt pair, where $\oomega:=\overline{\omega}$. Since $X$ is $\mbQ$-factorial, and $B+\bbeta_X+\boldsymbol{\omega}_X$ is big, by Corollary \ref{c-tran}, $\alpha$ is semiample, i.e. there is a proper surjective morphism $f:X\to Z$ to a normal compact K\"ahler variety $Z$ such that $\alpha=f^*\omega_Z$ for some K\"ahler class $\omega_Z\in H^{1,1}_{\BC}(Z)$. Since $X$ is projective, this implies that there is a curve $C\subset X$ such that $\alpha\cdot C=0$.
 By Lemma \ref{l-gklt2}(2), for any $\epsilon>0$, we may write 
\[K_X+B+\bbeta _X+\epsilon \omega \equiv K_X+\Delta_\epsilon  +\delta_\epsilon,\] 
where $\Delta_\epsilon \geq 0$ is a big $\R$-Cartier divisor such that $(X, \Delta_\eps)$ is klt and $\delta_\epsilon \in T(X)$.
Also, by Lemma \ref{lem:trans-alg} we can write $(1-\epsilon) \omega \equiv A_\epsilon +\gamma_\epsilon$, where $A_\epsilon$ is an ample $\mbR$-Cartier divisor and $\gamma _\epsilon\in T(X)$. 
% By Corollary \ref{c-tran}, $\alpha$ is semiample. Since it is not K\"ahler, there is a curve $C\subset X$ such that $\alpha \cdot C=0$. 
Therefore, we have $(K_X+\Delta_\epsilon) \cdot C=(\alpha-(1-\eps)\omega)\cdot C =-A_\epsilon\cdot C<0$. Observe that $\alpha\equiv K_X+\Delta_\eps+A_\eps+(\delta_\eps+\gamma_\eps)$, and thus $K_X+\Delta_\eps+A_\eps$ is nef by Lemma \ref{lem:trans-alg}.
Then by the usual cone theorem for $(X,\Delta _\epsilon)$, there is a $(K_X+\Delta_\epsilon +A_\epsilon)$-trivial $(K_X+\Delta_\epsilon)$-negative extremal ray $R$. Therefore, we may assume that $R=\mathbb R^{\geq 0}[\Gamma]$, where $\Gamma\subset X$ is a rational curve such that $\alpha \cdot \Gamma=(K_X+\Delta_\epsilon +A_\epsilon)\cdot \Gamma=0$ and $0<-(K_X+\Delta_\epsilon )\cdot \Gamma\leq 2\dim X$. Note that for any $0<\eps'<\eps$, we have $(K_X+\Delta_{\eps'})\cdot \Gamma=(K_X+\Delta_{\eps'}+\delta_{\eps'})\cdot \Gamma=(K_X+\Delta_{\eps}+\delta_\eps)\cdot \Gamma-(\eps-\eps')\omega <0$, i.e. $R$ is a $(K_X+\Delta_{\eps'})$-negative extremal ray for every $0<\eps'<\eps$. Then by the usual cone theorem again, there is a rational curve $\Gamma_{\eps'}$ contained in $R$ such that $0<-(K_X+\Delta_{\eps'})\cdot \Gamma_{\eps'}\leq 2\dim X$.\\

% \hl{Since}\footnote{Om: This argument is using the uncountability of $(0,\eps)$ and it's shorter. Please check it carefully. CH: This argument is incorrect. It could be that $0<-(K_X+\Delta_{\eps'})\cdot \bar \Gamma_i\leq 2\dim X$ for $\frac 1{i+1}\leq \eps'\leq \frac 1i$.} there are at most countably many numerically distinct curve classes on $X$, and $\{\Gamma_{\eps'}\}_{\eps'\in (0,\eps)}$ is an uncountable set, it follows that there is a rational curve $\bar\Gamma:=\Gamma_{\eps_0}$ in this collection for some $\eps_0\in (0,\eps)$ such that $0<-(K_X+\Delta_{\eps'})\cdot \bar \Gamma\leq 2\dim X$ holds for infinitely many values of $\eps'$ as $\eps'\to 0^+$. Therefore for the purpose of taking limit as $\eps'\to 0^+$ we may assume that $0<-(K_X+\Delta_{\eps'})\cdot \bar \Gamma\leq 2\dim X$ for all $\eps'\in (0,\eps)$. \\

Let
$H$ be an ample Cartier divisor on $X$. Consider the set $\{H\cdot C\}\subset \mbZ^+$, where $C\subset X$ is a rational curve such that $[C]\in R$. There is a rational curve $\bar \Gamma$ such that $H\cdot\bar\Gamma$ is the minimum of $\{H\cdot C\}$. Thus for any rational curve $C\subset X$ with $[C]\in R$ we have $H\cdot C\geq H\cdot \bar\Gamma$, and so there is a rational number $\lambda_C:=\frac{H\cdot C}{H\cdot\bar\Gamma}\geq 1$ such that $[C]=\lambda_C[\bar\Gamma]$ for all such rational curves. Therefore, for each $\eps'\in (0,\eps)$, there is a $\lambda_{\eps'}\geq 1$ such that $[\Gamma_{\eps'}]=\lambda_{\eps'}[\bar\Gamma]$. In particular, we have $-(K_X+\Delta_{\eps'})\cdot \bar\Gamma\leq -(K_X+\Delta_{\eps'})\cdot \Gamma_{\eps'}\leq 2\dim X$ for all $\eps'\in (0, \eps)$. Thus we have
\[0< -(K_X+B+\bbeta _X)\cdot \bar\Gamma= -(K_X+\Delta_{\epsilon'})\cdot \bar\Gamma+\epsilon' \omega\cdot \bar\Gamma\leq 2\dim X+\epsilon' \omega\cdot \bar\Gamma.\] Taking the limit as $\epsilon '\to 0^+$, we get $0<-(K_X+B+\bbeta _X)\cdot \bar\Gamma\leq 2\dim X$.\\

Now if $B+\bbeta_X$ is big, then using Lemma \ref{l-gklt2}(2) and \cite[Exercise 5.8]{HK10} we see that $I$ is a finite set.

\end{proof}

  \begin{corollary}
      Let $(X,B)$ be a projective klt pair and $\beta\in H^{1,1}_{\rm BC}(X)$ a nef class such that $\beta \cdot C\geq 2\dim X$ for every curve $C\subset X$. Then $K_X+B+\beta$ is nef. 
  \end{corollary}  
\begin{proof}
Let $\theta$ be a closed positive $(1, 1)$ current with local potentials contained in the class $\beta\in H^{1,1}_{\BC}(X)$. Then $\boldsymbol{\theta}:=\overline{\theta}$ defines a b-(1,1) current on $X$, and from the negativity lemma it follows that $(X, B+\boldsymbol{\theta})$ is a generalized klt pair. Then the result follows immediately from the cone Theorem \ref{thm:cone} applied to $(X, B+\boldsymbol{\theta})$.
\end{proof}

The next result answers a question of Valentino Tosatti.
\begin{theorem}
Let $p:X\to Y$ be a proper surjective morphism with connected fibers between normal compact K\"ahler varieties. Let $B$ be an effective $\mbQ$-divisor on $X$ such that $(X, B)$ is a klt pair and $K_X+B\num p^*\omega$ for some K\"ahler form $\omega$ on $Y$. Then $K_X+B$ is a semi-ample $\Q$-divisor.
% Let $p:X\to Y$ be a morphism of normal compact K\"ahler varieties, $(X,B)$ a klt pair where $B$ is an effective $\Q$ divisor, $\omega$ a K\"ahler form on $Y$, and $K_X+B\equiv p^*\omega$.
% Then $K_X+B$ is semiample.
\end{theorem}
\begin{proof} 
Suppose that $\dim Y=0$, i.e. $K_X+B\equiv 0$, then there is an integer such that $m(K_X+B)$ is Cartier, and replacing $m$ by a multiple, we may assume that $\OO _X(m(K_X+B))\in {\rm Pic }^0(X)$.
In particular, the following set is non-empty
\[
V^0(\OO _X(m(K_X+B))):=\{P\in \Pic^0(X)\;|\; h^0(X,\mcO_X(m(K_X+B))\otimes P)>0\}\ne \emptyset.
\]
Let $f:\tilde X\to X$ be a resolution of singularities of $X$ and $K_{\widetilde X}+\Gamma=f^*(K_X+B)+E$, where $\Gamma$ and $E$ are effective $\mbQ$-divisors without a common components and $f_*\Gamma=B$ and $f_*E=0$. Clearly, we have $f_*\mcO_{\widetilde X}(m(K_{\widetilde X}+\Gamma))=\mcO_X(m(K_X+B))$. Now applying  \cite[Theorem D]{Wang21} to the Albanese morphism $\tilde a: \widetilde X\to \Alb(\widetilde X)$ we see that $V^0(\mcO_{\widetilde X}(m(K_{\widetilde X}+\Gamma)))$ contains a torsion element of $\Pic^0(\widetilde X)$. Since $X$ has rational singularities, from the proof of \cite[Lemma 8.1]{Kaw85} it follows that $f^*:\Pic^0(X)\to \Pic^0(\widetilde X)$ is an isomorphism, and thus $V^0(\OO _X(m(K_X+B)))$ contains a torsion element of $\Pic^0(X)$. In particular, replacing $m$ by a higher multiple we may assume that $H^0(X, \mcO_X(m(K_X+B)))\neq 0$. However, since $K_X+B\num 0$, it follows that $m(K_X+B)\sim 0$. 

% it follows that $V^0(\OO _X(m(K_X+B)))$ contains a torsion element of ${\rm Pic }^0(X)$. Thus replacing $m$ by a multiple $H^0(\OO _X(m(K_X+B)))\ne 0$ and since $K_X+B\equiv 0$, then $m(K_X+B)\sim 0$.

Assume now that $\dim Y>0$. 
Note that it suffices to show that $K_X+B\sim _\Q p^* A$ for some $\Q$-Cartier divisor $A$ on $Y$; indeed, if this is the case, then from injectivty of $p^*: H^{1,1}_{\BC}(Y)\to H^{1,1}_{\BC}(X)$ it follows that $c_1(A)=[\omega]$ in $H^{1,1}_{\BC}(Y)$, and hence $A$ is ample. Now, since this question is local over $Y$, we may assume from now on that $Y$ is a relatively compact Stein open set.
Let $F$ be a general fiber. By adjunction $K_F+B_F=(K_X+B)|_F \equiv 0$, then by what we have just proved $K_F+B_F\sim _\Q 0$. Since $\C$ is uncountable, $p_*\OO _X(m(K_X+B))\ne 0$ for some $m>0$. 
Then replacing $m$ by a higher multiple and using \cite[Theorem 1.3]{DHP22} we may assume that $R(X/Y,K_X+B)$ is generated in degree $m$. Let $Y':={\rm Proj}R(X/Y,K_X+B)$ and $\nu: Y'\to Y$ the canonical model. Let $p':X\dasharrow Y'$ be the induced meromorphic map, then for some resolution $\mu :X'\to X$, we may assume that $q:X'\to Y'$ is a morphism and that $\mu ^*|m(K_X+B)|=|M|+G$, where $|M|$ is base point free and $\mcO_{X'}(M)\cong  q^* \OO _{Y'}(1)$.
Note that for any fixed $r$ and any $k\gg 0$ we have \[\nu _* \OO _{Y'}(k)\cong  (\nu \circ q)_*\OO _{X'}(kM)\cong (\nu \circ q)_*\OO _{X'}(kM+rG)\cong  \nu _*(q_*\OO _{X'}(rG)\otimes \OO _{Y'}(k)).   \]
Since $\OO _{Y'}(1)$ is ample over $Y$, it follows from the above relation that $q_*\OO _{X'}(rG)\cong \OO _{Y'}$. In particular, $G$ is $q$-exceptional, i.e.  given any irreducible component $P$ of $G$ such that $Q=q(P)$ is a divisor on $Y$, then there is a divisor $P'\not \subset {\rm Supp}(G)$ such that $Q=q(P')$. Since $\mcO_{X'}(M)\cong q^*\OO _{Y'}(1)$, then \[G\sim _{Y'}G+M=\mu ^*(m(K_X+B))\equiv \mu ^*p^*\omega \equiv _{Y'}0,\]  and hence by Lemma \ref{l-1}, $G=0$. Let $C$ be a $\mu$-exceptional curve which is not $q$-exceptional, then \[0=m\mu ^*(K_X+B)\cdot C=M\cdot C>0\] which is impossible. Thus every $\mu$-exceptional curve is $q$-exceptional and so by the rigidity lemma, $X\to Y'$ is a morphism, i.e. we may assume that $X=X'$, $G=0$, and $p^*\omega \equiv m(K_X+B)\sim {p'}^*\OO _{Y'}(1)$. Since $\OO _{Y'}(1)$ is ample over $Y$, again by the rigidity lemma we have $Y\cong Y'$. Consequently, $L:=p_* \OO _X(m(K_X+B))$ is a line bundle on $Y$ and $\mcO_X(m(K_X+B))\cong p^*L$. This completes our proof.
\end{proof}

\bibliographystyle{hep}
\bibliography{4foldreferences}

\end{document}